\documentclass[12pt,a4paper]{amsart}
\usepackage[english]{babel}
\usepackage{amsmath,amssymb,amsthm,latexsym,MnSymbol}
\usepackage[utf8]{inputenc}
\usepackage{tikz-cd}
\usepackage[colorlinks=true]{hyperref}
\usepackage[a4paper, portrait]{geometry}

\newcommand{\Z}{\mathbf{Z}}

\newcommand{\Q}{\mathbf{Q}}
\newcommand{\Qb}{\overline{\Q}}

\newcommand{\CC}{\mathbf{C}}
\newcommand{\h}{\mathcal{H}}
\newcommand{\A}{\mathbf{A}}

\newcommand{\T}{\mathbf{T}}

\newcommand{\et}{\textrm{ét}}

\DeclareMathOperator{\End}{End}
\DeclareMathOperator{\Gal}{Gal}
\DeclareMathOperator{\GL}{GL}
\DeclareMathOperator{\Hom}{Hom}

\DeclareMathOperator{\Spec}{Spec}

\newtheorem{thm}{Theorem}
\newtheorem*{thm*}{Theorem}
\newtheorem{lem}{Lemma}
\newtheorem{pro}[lem]{Proposition}

\newtheorem*{cor*}{Corollary}

\theoremstyle{definition}
\newtheorem{definition}[lem]{Definition}

\theoremstyle{remark}
\newtheorem{remark}{Remark}

\newtheorem*{remarks*}{Remarks}

\newtheorem*{questions}{Questions}

\begin{document}

\title[Modularity of endomorphism algebras]{On the modularity of endomorphism algebras}

\author[F. Brunault]{François Brunault}

\date{\today}

\address{ÉNS Lyon, Unité de mathématiques pures et appliquées, 46 allée d'Italie, 69007 Lyon, France}

\email{francois.brunault@ens-lyon.fr}
\urladdr{http://perso.ens-lyon.fr/francois.brunault}

\subjclass[2010]{Primary 11F41; Secondary 11F25, 11F70, 11F80, 14G32}
\keywords{Modular curves, Hecke correspondences, endomorphism algebras, automorphic representations, Galois representations}

\begin{abstract}
We use the adelic language to show that any homomorphism between Jacobians of modular curves arises from a linear combination of Hecke modular correspondences. The proof is based on a study of the actions of $\GL_2$ and Galois on the étale cohomology of the tower of modular curves. We also make this result explicit for Ribet's twisting operators on modular abelian varieties.
\end{abstract}

\maketitle

It is natural to ask whether the endomorphism algebra of the Jacobian of a modular curve is generated by the Hecke operators. Ribet showed in \cite{ribet:semistable} that if $N$ is prime, the algebra $(\End J_0(N)) \otimes \Q$ is generated by the Hecke operators $T_n$ with $n$ prime to $N$, answering positively a question of Shimura. Mazur \cite{mazur:eisenstein} subsequently showed an integral refinement of Ribet's result, namely that the algebra $\End J_0(N)$ is generated by the Hecke operators $T_p$ ($p$ prime, $p \neq N$) and by the Atkin-Lehner involution $w_N$.

For general $N$, the naive generalization of Ribet's result does not hold, since the Hecke operators generate a commutative subalgebra, while $\End J_0(N)$ is not commutative in general. The reason behind this is the existence of old modular forms, and we have to account for them if we want to find generators of the endomorphism algebra. In this direction, Kani \cite{kani} showed that if $\Gamma$ is a congruence subgroup such that $\Gamma_1(N) \subset \Gamma \subset \Gamma_0(N)$ with $N$ arbitrary, and $J_\Gamma$ is the Jacobian of the modular curve associated to $\Gamma$, then the algebra $\End(J_\Gamma) \otimes \Q$ is generated by the Hecke operators together with explicit degeneracy operators.

The purpose of this note is to develop an alternative approach to these questions using the adelic language. We show that after tensoring with $\Q$, any homomorphism between Jacobians of modular curves arises from a finite linear combination of Hecke modular correspondences. The cost of our abstract approach is that our results are less explicit in nature: we don't give explicit generators. On the other hand, our results are more general in that we consider homomorphisms instead of endomorphisms (Theorem \ref{thm1}), and in that we allow for homomorphisms defined over abelian extensions of $\Q$ (Theorem \ref{thm2}). I hope to convince the reader that the adelic language provides a convenient point of view for studying these questions.

Ribet showed in \cite{ribet:twists} that the endomorphism algebra of a modular abelian variety $A_f$ is generated over the Hecke field of $f$ by a finite set of endomorphisms coming from inner twists of $f$. In the last section of this paper, we explain how to write these endomorphisms in terms of Hecke correspondences, thus giving some substance to Theorem \ref{thm2}.

I thank Gabriel Dospinescu, Filippo A. E. Nuccio and Vincent Pilloni for interesting discussions related to this paper, and Eknath Ghate for useful advice.

\section{Statement of the main result}

Let $\A_f$ denote the ring of finite adèles of $\Q$, and let $G=\GL_2(\A_f)$. For any compact open subgroup $K$ of $G$, let $M_K$ denote the open modular curve over $\Q$ associated to $K$, and let $\overline{M}_K$ denote the compactification of $M_K$. If $M_K$ is geometrically connected, we denote by $J_K$ the Jacobian variety of $\overline{M}_K$.

Let $K,K'$ be compact open subgroups of $G$. We denote by $\tilde{\T}_{K,K'} = \Z[K \backslash G / K']$ the space of $\Z$-valued functions on $K\backslash G/K'$ with finite support. Assume that $M_K$ and $M_{K'}$ are geometrically connected. We then have a canonical map $\rho_J : \tilde{\T}_{K,K'} \to \Hom(J_K,J_{K'})$ (see Section \ref{adelic}).

\begin{thm}\label{thm1}
Let $K,K'$ be compact open subgroups of $G$ such that the modular curves $M_K$ and $M_{K'}$ are geometrically connected. Then $\rho_J(\tilde{\T}_{K,K'}) \otimes \Q = \Hom(J_K,J_{K'}) \otimes \Q$.
\end{thm}

\begin{remark}\label{rmk1}
Let $X$ be a smooth projective curve, and let $J$ be the Jacobian of $X$. It is known that every endomorphism of $J$ arises from an effective linear combination of correspondences on $X$. In the case $X$ is a modular curve, a general correspondence on $X$ could arise from a cover associated to a noncongruence subgroup. Our result says that congruences subgroups are enough to generate the endomorphism algebra.
\end{remark}

\begin{remark}\label{rmk2}
According to the Langlands philosophy, the Galois representations associated to algebraic varieties are expected to be automorphic. In fact, this conjectural correspondence should be functorial: not only the Galois representations, but also the morphisms between them should have an automorphic explanation. Theorem \ref{thm1} can be seen as a very simple case of this principle.
\end{remark}

\section{Modular curves in the adelic setting}\label{adelic}

Let $K$ be a compact open subgroup of $G=\GL_2(\A_f)$. The complex points of the modular curve $M_K$ are given by
\begin{equation*}
M_K(\CC) = \GL_2^+(\Q) \backslash (\h \times G) / K
\end{equation*}
where $\GL_2^+(\Q)$ acts on $\h \times G$ by $\gamma \cdot (\tau,g) = (\gamma(\tau),\gamma g)$, and $K$ acts on $G$ by right multiplication.

The set of connected components of $M_K(\CC)$ is in bijection with $\hat{\Z}^\times/\det(K)$. More precisely, let $\chi : \Gal(\Q^{\mathrm{ab}}/\Q) \xrightarrow{\cong} \hat{\Z}^\times$ denote the cyclotomic character, and let $F$ be the finite abelian extension of $\Q$ associated to $\chi^{-1}(\det(K))$. Then the structural morphism $M_K \to \Spec \Q$ factors through $\Spec F$, and the curve $M_K$ over $\Spec F$ is geometrically connected. We refer to $F$ as the \emph{base field of $M_K$}.

Let $K,K'$ be compact open subgroups of $G$, and let $g \in G$. We define a correspondence $\tilde{T}(g)$ between $M_{K}$ and $M_{K'}$ by the following diagram
\begin{equation*}
\begin{tikzcd}
& M_{K \cap g K' g^{-1}} \arrow{dl}[swap]{\alpha} \arrow{dr}{\alpha'} \\
M_K \arrow[dashed]{rr}{\tilde{T}(g)} & & M_{K'}
\end{tikzcd}
\end{equation*}
given on the complex points by $\alpha([\tau,h])=[\tau,h]$ and $\alpha'([\tau,h]) = [\tau,hg]$. The correspondence $\tilde{T}(g)$ extends to the compactifications and induces a map
\begin{equation*}
T(g) = \alpha'_* \circ \alpha^* : \Omega^1(\overline{M}_K) \to \Omega^1(\overline{M}_{K'}).
\end{equation*}
The map $T(g)$ depends only on the double coset $KgK'$, and there is a canonical map
\begin{equation*}
\rho_\Omega : \tilde{\T}_{K,K'} \to \Hom_\Q(\Omega^1(\overline{M}_K),\Omega^1(\overline{M}_{K'}))
\end{equation*}
sending the characteristic function of $KgK'$ to $T(g)$. We let $\T_{K,K'} = \rho_\Omega(\tilde{\T}_{K,K'})$.

Assume that $M_K$ and $M_{K'}$ are geometrically connected. For any $g \in G$, we define similarly $T(g) = \alpha'_* \circ \alpha^* : J_K \to J_{K'}$. Note that the homomorphism $T(g)$ is a priori defined over the base field of $M_{K \cap gK'g^{-1}}$, but its differential at the origin maps the tangent space $\Omega^1(\overline{M}_K)$ into $\Omega^1(\overline{M}_{K'})$, hence it is defined over $\Q$. We therefore get a map $\rho_J : \tilde{\T}_{K,K'} \to \Hom(J_K,J_{K'})$. Since $\Hom(J_K,J_{K'})$ acts faithfully on the tangent spaces, the map $\rho_J$ factors through $\T_{K,K'}$. Summing up, we have a commutative diagram
\begin{equation*}
\begin{tikzcd}
\tilde{\T}_{K,K'} \arrow{r} \arrow[bend left=15]{rrr}{\rho_\Omega} & \T_{K,K'} \arrow{r}{\rho_J} & \Hom(J_K,J_{K'}) \arrow{r}{\lambda} & \Hom_\Q(\Omega^1(\overline{M}_K),\Omega^1(\overline{M}_{K'}))
\end{tikzcd}
\end{equation*}
where $\lambda$ denotes the differential at the origin.

In the case $K=K'$, we put $\tilde{\T}_K = \tilde{\T}_{K,K}$ and $\T_K = \T_{K,K}$. The convolution product endows $\tilde{\T}_{K}$ with the structure of a (unitary) ring, and $\tilde{\T}_{K,K'}$ is a $(\tilde{\T}_{K},\tilde{\T}_{K'})$-bimodule. Note that $\T_{K,K'}$ is a $(\T_K,\T_{K'})$-bimodule.

\section{Proof of Theorem \ref{thm1}}

Define
\begin{equation*}
\Omega = \varinjlim_K \Omega^1(\overline{M}_K) \otimes \Qb,
\end{equation*}
where the direct limit is taken with respect to the pull-back maps. The space $\Omega$ is endowed with an action of $G$ and the subspace $\Omega^K$ of $K$-invariants coincides with $\Omega^1(\overline{M}_K) \otimes \Qb$. According to the multiplicity one theorem, the space $\Omega$ decomposes as a direct sum of distinct irreducible admissible representations of $G$:
\begin{equation*}
\Omega = \bigoplus_{\pi \in \Pi} \Omega(\pi).
\end{equation*}
Let $\Pi(K)$ be the set of those representations $\pi \in \Pi$ satisfying $\Omega(\pi)^K \neq 0$. We have a direct sum decomposition
\begin{equation*}
\Omega^1(\overline{M}_K) \otimes \Qb = \bigoplus_{\pi \in \Pi(K)} \Omega(\pi)^K,
\end{equation*}
and the spaces $\Omega(\pi)^K$ are pairwise non-isomorphic simple $\T_K \otimes \Qb$-modules \cite[p. 393]{langlands}.

\begin{lem}\label{iso TK}
The canonical map
\begin{equation*}
\rho_{K} : \T_K \otimes \Qb \to \prod_{\pi \in \Pi(K)} \End_{\Qb}(\Omega(\pi)^K)
\end{equation*}
is an isomorphism.
\end{lem}

\begin{proof}
The map $\rho_{K}$ is injective by definition of $\T_K$. The surjectivity follows from Burnside's Theorem \cite[\S 5, N°3, Cor. 1 of Prop. 4, p. 79]{bourbaki-alg8}.
\end{proof}

\begin{lem}\label{iso TKK'}
Let $K,K'$ be compact open subgroups of $G$. For any $\pi \in \Pi$, the bimodule $\T_{K,K'}$ maps $\Omega(\pi)^K$ into $\Omega(\pi)^{K'}$. Let $\mathcal{R}=\Pi(K) \cap \Pi(K')$. The map
\begin{equation*}
\rho_{K,K'} : \T_{K,K'} \otimes \Qb \to \prod_{\pi \in \mathcal{R}} \Hom_{\Qb}(\Omega(\pi)^K,\Omega(\pi)^{K'})
\end{equation*}
is an isomorphism of $(\T_K,\T_{K'})$-bimodules.
\end{lem}

\begin{proof}
The map $\rho_{K,K'}$ is injective by definition of $\T_{K,K'}$. For the surjectivity, let $K''$ be a compact open subgroup of $G$ such that $\mathcal{R} \subset \Pi(K'')$. We have a commutative diagram
\begin{equation*}
\begin{tikzcd}
\T_{K,K''} \otimes \T_{K'',K'} \otimes \Qb \arrow{r} \arrow{d} & \prod_{\pi \in \mathcal{R}} \Hom(\Omega(\pi)^K, \Omega(\pi)^{K''}) \otimes \Hom(\Omega(\pi)^{K''},\Omega(\pi)^{K'}) \arrow{d} \\
\T_{K,K'} \otimes \Qb \arrow{r} & \prod_{\pi \in \mathcal{R}} \Hom(\Omega(\pi)^K,\Omega(\pi)^{K'}).
\end{tikzcd}
\end{equation*}
Since the right-hand map is surjective, it suffices to show that the maps $\rho_{K,K''}$ and $\rho_{K'',K'}$ are surjective. Choosing $K''=K \cap K'$, we are reduced to show the lemma in the cases $K' \subset K$ and $K \subset K'$. Moreover, since $\T_{K,K'}$ is a $(\T_K,\T_{K'})$-bimodule, and thanks to Lemma \ref{iso TK}, it suffices to show that for any $\pi \in \mathcal{R}$, the map
\begin{equation*}
\rho_\pi : \T_{K,K'} \otimes \Qb \to \Hom_{\Qb}(\Omega(\pi)^K,\Omega(\pi)^{K'})
\end{equation*}
is non-zero.

In the case $K' \subset K$, the image of the double coset $K \cdot 1 \cdot K' = K$ under $\rho_\pi$ is the inclusion map of $\Omega(\pi)^K$ into $\Omega(\pi)^{K'}$, which is non-zero.

In the case $K \subset K'$, the image of the double coset $K \cdot 1 \cdot K' = K'$ under $\rho_\pi$ is the trace map from $\Omega(\pi)^{K}$ to $\Omega(\pi)^{K'}$. Since the restriction of the trace map to $\Omega(\pi)^{K'}$ is the multiplication by the index $(K':K)$, the trace map is non-zero as required.
\end{proof}

Now let us consider the direct limit of the étale cohomology groups of $\overline{M}_K$:
\begin{equation*}
H = \varinjlim_K H^1_\et(\overline{M}_K \otimes_\Q \Qb,\Z_\ell) \otimes \Qb_\ell.
\end{equation*}
The space $H$ is endowed with two commuting actions of $G$ and $\Gamma_\Q = \Gal(\Qb/\Q)$, and we have
\begin{equation*}
H^K = H^1_\et(\overline{M}_K \otimes_\Q \Qb,\Z_\ell) \otimes \Qb_\ell.
\end{equation*}

We will now see how to ``separate'' these two actions. Let us fix an embedding of $\Qb$ into $\Qb_\ell$.

\begin{definition}
For any $\pi \in \Pi$, let $V_\pi = \Hom_{G}(\Omega(\pi),H)$.
\end{definition}

Note that $V_\pi$ is a $\Qb_\ell$-vector space endowed with an action of $\Gamma_\Q$.

\begin{lem}\label{decomp HK}
The Galois representation $V_\pi$ is $2$-dimensional, and we have a $G \times \Gamma_\Q$-equivariant isomorphism
\begin{equation*}
H \cong \bigoplus_{\pi \in \Pi} \Omega(\pi) \otimes_{\Qb} V_\pi.
\end{equation*}
In particular, for any compact open subgroup $K$ of $G$, we have a $\T_K[\Gamma_\Q]$-equivariant isomorphism
\begin{equation*}
H^K \cong \bigoplus_{\pi \in \Pi(K)} \Omega(\pi)^K \otimes_{\Qb} V_\pi.
\end{equation*}
\end{lem}

\begin{proof}
Let us fix an isomorphism $\Qb_\ell \cong \CC$. By the comparison theorem between Betti and étale cohomology, we have
\begin{equation*}
H \cong \varinjlim_K H^1_B(\overline{M}_K(\CC),\Qb_\ell).
\end{equation*}
On the other hand, we have a $\CC$-linear isomorphism
\begin{align*}
\Omega^1(\overline{M}_K(\CC)) \oplus \Omega^1(\overline{M}_K(\CC)) & \xrightarrow{\cong} H^1_B(\overline{M}_K(\CC),\CC) \\
(\omega,\omega') & \mapsto [\omega + c^* \omega']
\end{align*}
where $c$ denotes the complex conjugation on $\overline{M}_K(\CC)$. It follows that $H \cong (\Omega \oplus \Omega) \otimes \Qb_\ell$. Since
\begin{equation*}
\Hom_{G}(\Omega(\pi),\Omega(\pi')) = \begin{cases} \Qb & \textrm{if } \pi = \pi' \\
0 & \textrm{if } \pi \neq \pi',
\end{cases}
\end{equation*}
we deduce that $V_\pi$ has dimension 2. Finally, there is a canonical map $\Omega(\pi) \otimes V_\pi \to H$, and the space $H$ decomposes as the direct sum of the images of these maps.
\end{proof}

The following lemma is well-known (see the proof of \cite[Thm 4.4]{ribet:twists}).

\begin{lem}\label{rep Vpi}
The representation $V_\pi$ is irreducible, and we have
\begin{equation*}
\Hom_{\Gamma_\Q}(V_\pi,V_{\pi'})= \begin{cases} \Qb_\ell & \textrm{if } \pi \cong \pi',\\
0 & \textrm{if }  \pi \not\cong \pi'. \end{cases}
\end{equation*}
\end{lem}

\begin{proof}[Proof of the main theorem]
Let $K,K'$ be compact open subgroups of $G$ such that $M_K$ and $M_{K'}$ are geometrically connected. Consider the composite map
\begin{equation*}
\begin{tikzcd}
\T_{K,K'} \otimes \Qb_\ell \arrow{r}{\rho_J \otimes 1} \arrow[bend left=15]{rr}{\rho_\et} & \Hom(J_K,J_{K'}) \otimes \Qb_\ell \arrow{r}{\mu} & \Hom_{\Gamma_\Q}(H^K,H^{K'}).
\end{tikzcd}
\end{equation*}
Since these maps are injective, it suffices to show that $\rho_\et$ is surjective, and for this it is enough to compare the dimensions. Let $\mathcal{R}=\Pi(K) \cap \Pi(K')$. By Lemma \ref{iso TKK'}, we have
\begin{equation*}
\dim \T_{K,K'} = \sum_{\pi \in \mathcal{R}} (\dim \Omega(\pi)^K) (\dim \Omega(\pi)^{K'}).
\end{equation*}
On the other hand, using Lemmas \ref{decomp HK} and \ref{rep Vpi}, we get
\begin{equation*}
\Hom_{\Gamma_\Q} (H^K,H^{K'}) = \bigoplus_{\pi \in \mathcal{R}} \Hom(\Omega(\pi)^K,\Omega(\pi)^{K'}) \otimes \Qb_\ell,
\end{equation*}
and thus $\dim \Hom_{\Gamma_\Q} (H^K,H^{K'}) = \dim \T_{K,K'}$ as desired.
\end{proof}

\section{Generalization to abelian extensions}

Let $K$ be a compact open subgroup of $G$, and let $F$ be the base field of $M_K$. Let $F'$ be a finite abelian extension of $\Q$ containing $F$, and let $U_{F'}$ be the subgroup of $U=\det(K)$ defined by $U_{F'} = \chi(\Gal(\Q^{\mathrm{ab}}/F'))$. Then we have a canonical isomorphism $M_K \otimes_F F' \cong M_{K_{F'}}$ where $K_{F'}$ is the subgroup of $K$ defined by
\begin{equation*}
K_{F'} = \{g \in K : \det(g) \in U_{F'}\}.
\end{equation*}

Let $K,K'$ be compact open subgroups of $G$ such that the base fields of $M_K$ and $M_{K'}$ are equal to a fixed finite abelian extension $F$ of $\Q$.

\begin{definition}
Let $T=(X,\alpha,\alpha')$ be a finite correspondence between $\overline{M}_K$ and $\overline{M}_{K'}$, seen as curves over $\Q$. We say that \emph{$T$ is defined over $F$} if the following diagram commutes:
\begin{equation*}
\begin{tikzcd}
& X \arrow{dl}[swap]{\alpha} \arrow{dr}{\alpha'} \\
\overline{M}_K \arrow{dr}[swap]{\delta} \arrow[dashed]{rr}{T} & & \overline{M}_{K'} \arrow{dl}{\delta'}\\
& \Spec F
\end{tikzcd}
\end{equation*}
\end{definition}

\begin{lem}\label{lem Tg defined over F}
Let $U_F = \chi(\Gal(\Q^{\mathrm{ab}}/F))$ and let $g \in G$. The correspondence $\tilde{T}(g)=KgK'$ is defined over $F$ if and only if $\det(g) \in \Q_{>0} \cdot U_F$.
\end{lem}

We denote by $\T_{K,K';F}$ the subgroup of $\T_{K,K'}$ generated by those correspondences $T(g)$ which are defined over $F$. Note that we have a canonical map $\rho_J : \T_{K,K';F} \to \Hom_F(J_K,J_{K'})$ where $J_K$ (resp. $J_{K'}$) denotes the Jacobian variety of $\overline{M}_K$ (resp. $\overline{M}_{K'}$) over $F$.

\begin{thm}\label{thm2}
Let $K,K'$ be compact open subgroups of $\GL_2(\A_f)$, and let $F$ be a finite abelian extension of $\Q$ containing the base fields of $M_K$ and $M_{K'}$. Then the canonical map
\begin{equation*}
\rho_J : \T_{K_F,K'_F;F} \otimes \Q \to \Hom_F(J_K,J_{K'}) \otimes \Q
\end{equation*}
is an isomorphism.
\end{thm}

\begin{proof}
By the above discussion, it is sufficient to prove the theorem in the case the base fields of $M_K$ and $M_{K'}$ are equal to $F$. Let $\Gamma = \Gal(F/\Q)$. For any $\gamma \in \Gamma$, let $\T_{K,K';\gamma}$ denote the subgroup of $\T_{K,K'}$ generated by those correspondences $T(g)$ satisfying $\det(g) \in \Q_{>0} \cdot (\hat{\gamma}U_F)$, where $\hat{\gamma} \in \hat{\Z}^\times$ is any element satisfying $\chi^{-1}(\hat{\gamma})|_F = \gamma$. Since the elements of $\T_{K,K';\gamma}$ are $\gamma$-linear, we have a direct sum decomposition
\begin{equation*}
\T_{K,K'}  = \bigoplus_{\gamma \in \Gamma} \T_{K,K';\gamma}.
\end{equation*}
By the proof of Theorem \ref{thm1}, we have an isomorphism
\begin{equation}\label{iso rho_et}
\rho_\et : \T_{K,K'} \otimes \Qb_\ell \xrightarrow{\cong} \Hom_{\Gamma_\Q}(H^K,H^{K'}).
\end{equation}
We now wish to identify those elements of $\Hom_{\Gamma_\Q}(H^K,H^{K'})$ which come from $\T_{K,K';\gamma}$. Let $\Sigma$ denote the set of embedding of $F$ into $\Qb$. We have
\begin{equation*}
\overline{M}_K \otimes_\Q \Qb = \bigsqcup_{\sigma \in \Sigma} \overline{M}_K \otimes_{F,\sigma} \Qb
\end{equation*}
inducing a direct sum decomposition $H^K = \bigoplus_{\sigma \in \Sigma} H^K_\sigma$ with
\begin{equation*}
H^K_\sigma = H^1_\et(\overline{M}_K \otimes_{F,\sigma} \Qb, \Z_\ell) \otimes \Qb_\ell.
\end{equation*}
Note that the action of $\Gamma_\Q$ on $H^K$ permutes the components $H^K_\sigma$ according to the rule $\gamma \cdot H^K_\sigma = H^K_{\gamma \sigma}$ for any $\gamma \in \Gamma_\Q$. Fixing an element $\sigma_0 \in \Sigma$, we have an isomorphism
\begin{equation*}
\operatorname{Ind}_{\Gamma_F}^{\Gamma_\Q} H^K_{\sigma_0} \cong H^K
\end{equation*}
where $\Gamma_F = \Gal(\Qb/F)$. By Frobenius reciprocity, we have
\begin{equation}\label{iso Frob}
\Hom_{\Gamma_\Q}(H^K,H^{K'}) \cong \Hom_{\Gamma_F}(H^K_{\sigma_0},H^{K'}).
\end{equation}
Moreover $\T_{K,K';\gamma}$ maps $H^K_\sigma$ into $H^{K'}_{\sigma \gamma}$. Combining the isomorphisms (\ref{iso rho_et}) and (\ref{iso Frob}), we get
\begin{equation*}
\T_{K,K';\gamma} \otimes \Qb_\ell \cong \Hom_{\Gamma_F}(H^K_{\sigma_0},H^{K'}_{\sigma_0 \gamma}) \qquad (\gamma \in G).
\end{equation*}
Taking $\gamma=1$, we get a commutative diagram
\begin{equation*}
\begin{tikzcd}
\T_{K,K';F} \otimes \Qb_\ell \arrow{r}{\rho_J \otimes 1} \arrow[bend left=15]{rr}{\rho_\et} & \Hom_F(J_K,J_{K'}) \otimes \Qb_\ell \arrow{r}{\mu} & \Hom_{\Gamma_F}(H^K_{\sigma_0},H^{K'}_{\sigma_0})
\end{tikzcd}
\end{equation*}
where $\rho_\et$ is an isomorphism. We conclude as in the proof of Theorem \ref{thm1}.
\end{proof}

We now give some consequences for endomorphism algebras of modular abelian varieties.

\begin{cor*}\label{cor}
Let $K$ be a compact open subgroup of $G$. Let $F$ be a finite abelian extension of $\Q$ containing the base field of $M_K$.

\begin{enumerate}
\item Let $A/F$ be an abelian subvariety of $J_K/F$. Define
\begin{equation*}
\T_A = \{T \in \T_{K_F;F} : \rho_J(T) \textrm{ leaves stable } A\}.
\end{equation*}
Then the canonical map $\T_A \otimes \Q \to \End_F(A) \otimes \Q$ is surjective.
\item Let $A/F$ be an abelian variety which is a quotient of $J_K/F$. Define
\begin{equation*}
\T^A = \{T \in \T_{K_F;F} : \rho_J(T) \textrm{ factors through } A\}.
\end{equation*}
Then the canonical map $\T^A \otimes \Q \to \End_F(A) \otimes \Q$ is surjective.
\end{enumerate}
\end{cor*}

\begin{proof}
Let us prove (1). Let $\iota : A \to J_K$ denote the inclusion map. Let $p : J_K \to A$ be a homomorphism such that $p \circ \iota = [n]_A$ for some integer $n \neq 0$. Let $\phi \in \End_F(A)$. Define $\psi = \iota \circ \phi \circ p \in \End_F(J_K)$. By Theorem \ref{thm2}, there exists $T \in \T_{K_F;F} \otimes \Q$ such that $\rho_J(T)=\psi$. Note that $\psi$ leaves stable $A$, so that $T \in \T_A \otimes \Q$, and we have $\psi |_A = [n] \phi$.

The proof of (2) is similar.
\end{proof}

We emphasize that the corollary is true even for elliptic curves with complex multiplication, as long as their endomorphisms are defined over an abelian extension of $\Q$. For example, the elliptic curve $E=X_0(32)$ has complex multiplication by $\Z[i]$ defined over $\Q(i)$. Let
\begin{equation*}
K = K_0(32)_{\Q(i)} = \left\{\begin{pmatrix} a & b \\ c & d \end{pmatrix} \in \GL_2(\hat{\Z}): c \equiv 0 \; (32), ad \equiv 1 \;(4) \right\}.
\end{equation*}
The matrix $\begin{pmatrix} 1 & 0 \\ 8 & 1 \end{pmatrix}$ normalizes $K$, and the canonical map $\T_{K;\Q(i)} \to \End_{\Q(i)} E \cong \Z[i]$ maps the double coset $K \begin{pmatrix} 1 & 0 \\ 8 & 1 \end{pmatrix} K = K \begin{pmatrix} 1 & 0 \\ 8 & 1 \end{pmatrix}$ to the element $\pm i$, hence $\T_{K;\Q(i)} = \End_{\Q(i)} E$.

To conclude this section, let me mention some open questions.

\begin{questions}
\begin{enumerate}
\item Do Theorems \ref{thm1} and \ref{thm2} hold integrally?
\item Do Theorems \ref{thm1} and \ref{thm2} hold in positive characteristic?
\item The analogue of $J_K$ in weight $k>2$ is the motive associated to the space of cusp forms of weight $k$ and level $K$ \cite{scholl:motivesMF}. Do the results presented here extend to these motives?
\end{enumerate}
\end{questions}

\section{Comparison with Ribet's result}

Let $f = \sum_{n \geq 1} a_n q^n$ be a newform of weight $2$ on $\Gamma_1(N)$, and let $A_f$ be the modular abelian variety over $\Q$ associated to $f$. The abelian variety $A_f$ is simple over $\Q$ and the algebra $\End_{\Q}(A_f) \otimes \Q$ is isomorphic to the Hecke field $K_f$ of $f$. Ribet determined in \cite{ribet:twists} the structure of the endomorphism algebra $\End_{\Qb}(A_f) \otimes \Q$. In particular, he proved that this algebra is generated over $K_f$ by finitely many endomorphisms coming from inner twists of $f$. Our goal in this section is to write these endomorphisms in terms of Hecke correspondences, making thus Theorem \ref{thm2} explicit for these endomorphisms.

Let us first recall Ribet's construction \cite[\S 5]{ribet:twists}. We assume that $f$ doesn't have complex multiplication. Let $h$ denote the modular form $h(z)=\sum_{(n,N)=1} a_n q^n$. Let $\Gamma$ denote the set of automorphisms $\gamma$ of $K_f$ such that $f^\gamma = f \otimes \chi_\gamma$ for some Dirichlet character $\chi_\gamma$. Let $m$ denote the least common multiple of $N$ and the conductors of the caracters $\chi_\gamma$. Then $h$ is an eigenform on the group $\Gamma_0(m^2) \cap \Gamma_1(m)$. Let $J$ denote the Jacobian variety of the modular curve associated to this group. By Shimura's construction \cite[\S 2]{ribet:twists}, there exists an optimal quotient $A_h$ of $J$ associated to $h$. Let $\nu : J \to A_h$ denote the canonical projection. The abelian varieties $A_f$ and $A_h$ are isogenous. In particular, their endomorphism algebras are isomorphic. 

For every $\gamma \in \Gamma$, Ribet constructs an endomorphism $\eta_\gamma$ of $A_h$ as follows. Write $f^\gamma = f \otimes \chi$. Let $r$ denote the conductor of $\chi$. For every $u \in \Z$, there is an endomorphism $\alpha_{u/r}$ of $J$ acting on the space of cusp forms as $g \mapsto g(z+u/r)$. Define
\begin{equation*}
\tilde{\eta}_\gamma = \sum_{u \in \Z/r\Z} \chi^{-1}(u) \circ \nu \circ \alpha_{u/r} \in \Hom(J,A_h) \otimes \Q
\end{equation*}
where $\chi^{-1}(u) \in K_f$ is seen as an element of $\End_{\Q}(A_h) \otimes \Q$. Then $\tilde{\eta}_\gamma$ factors through $\nu$ and induces an endomorphism $\eta_\gamma$ of $A_h$. Since $\alpha_{u/r}$ is defined over $\Q(\zeta_r)$, we have $\eta_\gamma \in \End_{\Q(\zeta_r)}(A_h) \otimes \Q$.

Let us now turn to the adelic language. Consider the group
\begin{equation*}
K = K_0(m^2) \cap K_1(m) = \left\{ \begin{pmatrix} a & b \\ c & d \end{pmatrix} \in \GL_2(\hat{\Z}) : c \equiv 0 (m^2), d \equiv 1 (m) \right \}
\end{equation*}
and its subgroup $K' = K_{\Q(\zeta_r)}$. The modular curve $\overline{M}_{K'}$ and its Jacobian $J' = J_{K'}$ are defined over the field $\Q(\zeta_r)$, and we have a canonical isomorphism $J' \cong J_{\Q(\zeta_r)}$.

\begin{lem}\label{lem alpha}
For every $u \in \Z$, we have $\alpha_{u/r} = \rho_{J'} \left(T \left(\begin{smallmatrix} 1 & u/r \\ 0 & 1 \end{smallmatrix} \right) \right)$.
\end{lem}

\begin{proof}
By Lemma \ref{lem Tg defined over F}, the correspondence $\tilde{T}\left(\begin{smallmatrix} 1 & u/r \\ 0 & 1 \end{smallmatrix}\right)$ is defined over $\Q(\zeta_r)$. Moreover the matrix $\left(\begin{smallmatrix} 1 & u/r \\ 0 & 1 \end{smallmatrix}\right)$ normalizes $K'$, so that $T\left(\begin{smallmatrix} 1 & u/r \\ 0 & 1 \end{smallmatrix}\right)$ acts on $\Omega^1(\overline{M}_{K'})$ by sending a cusp form $g$ to $g(z+u/r)$. It follows that $\alpha_{u/r}^* = \rho_\Omega \left(\tilde{T}\left(\begin{smallmatrix} 1 & u/r \\ 0 & 1 \end{smallmatrix}\right)\right)$, hence the Lemma.
\end{proof}

\begin{lem}\label{lem lambda}
For every $u \in \Z$, there exists $\lambda_u \in \T_{K';\Q(\zeta_r)} \otimes \Q$ such that $\rho_{J'}(\lambda_u)$ factors through $A_h$ and induces $\chi^{-1}(u)$ on $A_h$.
\end{lem}

\begin{proof}
This is a direct consequence of Corollary (2).
\end{proof}

We now define
\begin{equation*}
X_\gamma = \sum_{u=0}^{r-1} \lambda_u \cdot T \begin{pmatrix} 1 & u/r \\ 0 & 1 \end{pmatrix} \in \T_{K';\Q(\zeta_r)} \otimes \Q.
\end{equation*}

\begin{pro}
The endomorphism $\rho_{J'}(X_\gamma)$ factors through $A_h$ and induces the endomorphism $\eta_\gamma$ on $A_h$.
\end{pro}

\begin{proof}
This follows from Lemmas \ref{lem alpha} and \ref{lem lambda}.
\end{proof}

\bibliographystyle{plain}
\bibliography{references}

\end{document}